\documentclass{amsart}\usepackage{amscd, enumerate}
\usepackage{amssymb}
\usepackage{amscd}
\usepackage[all]{xy}

\usepackage{todonotes}

\numberwithin{equation}{section}

\newtheorem{theorem}{Theorem}[section]
\newtheorem{lemma}[theorem]{Lemma}
\newtheorem{proposition}[theorem]{Proposition}
\newtheorem{corollary}[theorem]{Corollary}

\theoremstyle{definition}

\newtheorem{example}[theorem]{Example}

\newtheorem{problem}[theorem]{Problem}

\numberwithin{equation}{section}

\def\nin{\noindent}

\def\a{{\alpha}}
\def\b{{\beta}}
\def\g{{\gamma}}

\def\k{{\kappa}}
\def\l{{\lambda}}
\def\f{{\phi}}
\def\r{{\rho}}
\def\s{{\sigma}}
\def\t{{\tau}}
\def\w{{\omega}}

\def\CC{{\mathcal C}}

\def\TT{{\mathcal T}}

\def\GG{{\mathcal G}}

\def\al{{\aleph}}

\def\rk{\mathop{\rm rk}}
\def\Hom{\mathop{\rm Hom}\nolimits}

\def\Ext{\mathop{\rm Ext}\nolimits}

\def\Im{\mathop{\rm Im}}
\def\fg{finitely generated}
\def\cg{countably generated}
\def\fp{finitely presented}
\def\cp{cyclically presented}

\def\tf{torsion-free}

\numberwithin{equation}{section}
\def\<{\langle}
\def\>{\rangle}

\long\def\alert#1{\smallskip{\hskip\parindent\vrule%
\vbox{\advance\hsize-2\parindent\hrule\smallskip\parindent.4\parindent%
\narrower\noindent#1\smallskip\hrule}\vrule\hfill}\smallskip}

\numberwithin{equation}{section}

\def\CC{{\mathcal C}}

\def\GG{{\mathcal G}}

\def\TT{{\mathcal T}}

\def\nin{\noindent}

\def\a{{\alpha}}
\def\b{{\beta}}
\def\g{{\gamma}}

\def\k{{\kappa}}
\def\l{{\lambda}}

\def\f{{\phi}}
\def\r{{\rho}}
\def\s{{\sigma}}
\def\t{{\tau}}
\def\w{{\omega}}

\def\al{{\aleph}}

\def\Hom{\mathop{\rm Hom}\nolimits}
\def\Ext{\mathop{\rm Ext}\nolimits}

\def\Im{\mathop{\rm Im}}
\def\End{\mathop{\rm End}\nolimits}
\def\Aut{\mathop{\rm Aut}\nolimits}

\def\Pext{\mathop{\rm Pext}\nolimits}

\def\ann{\mathop{\rm ann}\nolimits}

\def\gen{\mathop{\rm gen}}
\def\rk{\mathop{\rm rk}}

\def\fg{finitely generated}
\def\fp{finitely presented}

\def\tf{torsion-free}

\def\vd{valuation domain}

\def\vds{valuation domains}
\def\Dd{Dedekind domain}

\date\today

\begin{document}

\pagestyle{plain}


\font\bftwo=cmbx10 scaled\magstep{1.2}

\title {On tight submodules \\ of modules over valuation domains}


\author{Peter Danchev}
\address{Institute of Mathematics \& Informatics, Bulgarian Academy of Sciences, 1113 Sofia, Bulgaria}
\email{danchev@math.bas.bg; pvdanchev@yahoo.com}



\author{L\'aszl\'o Fuchs}
\address{Department of Mathematics, Tulane University, New Orleans, Louisiana 70118, USA
{\it and }  1724 Pinetree Cir NE, Atlanta, Georgia  30329, USA}
\email{fuchs@tulane.edu}



\begin{abstract} This note offers an unusual approach of studying  a class of modules inasmuch as it is investigating a subclass of the category of modules over a \vd. This class is far from being a full subcategory, it is not even a category. Our concern is the subclass consisting of modules of projective dimension not exceeding one,  admitting only morphisms whose kernels and cokernels are also  objects in this subclass. This class is still tractable, several features are  in general simpler than in module categories, but lots of familiar   properties are lost.   A number of results on modules in this class are similar to those on modules over rank one discrete \vds\ (where the global dimension is 1). The study is considerably simplified by taking advantage of the general theory of modules over \vds\ available in the literature, e.g. in \cite{FS1}-\cite{FS2}.   Our main goal is to establish the basic features and have a closer look at injectivity, pure-injectivity, and cotorsionness, but we do not wish to enter into an in-depth study of these properties.
\end{abstract}

\maketitle


\section{Introduction} 

Everybody familiar with module theory    over integral domains knows well that the theory simplifies tremendously if  the ring is a \Dd; i.e. the modules have projective dimension (p.d.) at most one.   That  the condition `p.d.1' is so powerful  was recognized by E. Matlis who developed interesting properties under the hypothesis that the field of quotients as a module has p.d.1; see \cite{M0} (in \cite{FS2}  these rings were named Matlis domains in his honor). In order to understand p.d.1 better and to learn more about it, it is natural to try to find out  how the theory would look like over a general  integral domain if one deals only with its modules of p.d. not exceeding $1$, and ignores modules of higher projective dimensions. This means  to investigate a class where the objects are modules of p.d.$\le 1$ and morphisms are required to have both kernels and cokernels of p.d.$\le 1$. This is what we are planning to do in this article over an arbitrary \vd\ (i.e. an integral domain where the ideals form a chain with respect to inclusion).  Valuation domains are the first and obvious choice for such  a study, as they are sufficiently general, but still manageable, and luckily, there is an extensive literature available on them that reveals a lot of information one can take advantage of.

The selected subclass of the module category is not a category;  a primary reason for it is that the usual composition rule of mappings works only under an additional condition. When dealing with modules of such a class, it is soon realized that one has to reassess familiar facts and obvious concepts to fit into the new situation. The usual sum of two morphisms is rarely another one, submodules that belong to this subclass only exceptionally form a lattice, the tensor product of two objects may not belong to this class, etc. But on the other hand, some nice features of discrete rank one \vds\ carry over, like the equality of injectivity and divisibility or the pure-injectivity of cyclically presented torsion modules.  We will also pay attention to pure-injectivity and to the cotorsion property (Sections 7 and 8). Though these concepts are defined to have close resemblance to the familiar module concepts, one should always keep in mind that they are not exactly the same.  The discussion of the submodules in direct sums of \cp\ modules (Section 5) already shows the huge difference from the traditional treatment of modules.
\medskip

Throughout the symbol $V$ will denote a valuation domain (commutative), and $Q$ its quotient field. We will use the notation $K = Q/V$.  The symbol $V^\times$ denotes the multiplicative monoid of non-zero elements of $V$. Torsion and torsion-free modules have the usual meaning.  The abbreviation p.d. denotes projective dimension. The global weak dimension of a valuation domain is $1$. Likewise, $|X|$ denotes the cardinality of the set $X$, and  $\w$ is the smallest infinite  ordinal. We also abbreviate gen~ $M$ to stand for the minimal cardinality of  generating sets of $M$. We use ``countably generated" to mean gen ~$M = \al_0$ (not finite).  Torsion-freeness and flatness are equivalent over \vds; consequently, relative divisibility and purity have the same meaning.
For the construction of valuation domains with prescribed value groups, see e.g. \cite[Chap. II, Section 3]{FS2}.

For a valuation domain $V$, we consider the category $\CC_V$ as the category $V$-Mod of $V$-modules with the usual morphisms. Our main goal is to study the subclass $\CC_V^*$ whose objects are the $V$-modules of p.d.$ \le 1$  and whose  morphisms are required to have kernels and  cokernels that also have p.d.$ \le 1$.   The subobjects are the tight submodules (see Section 2).  While our results are restricted to this subclass,  in proofs we will often use arguments and concepts from the covering category $\CC_V$.

 Modules of projective dimension one have been discussed briefly  in an earlier paper \cite{F7} with emphasis on Pr\"ufer domains, and  some of our present results appeared there in a different context.

 The idea of investigating the subclass $\CC_V^*$  comes from a recent paper \cite{F6} where tight submodules played a dominating role. To deal with classes of modules where both the kernels and the cokernels of the maps were restricted looked strange, but at the same time interesting and  challenging.  We admit, we were first hesitating to get involved in an uncharted territory with no immediate applications in the horizon, but in spite of this we decided to start working on this class,  since we believed that the results could be helpful in a better understanding of the impact of projective dimension one as well as of the role of tightness in submodules. In addition, they might provide counterexamples in unusual situations.  In this paper we begin with exploring this idea, and accordingly, we have been trying to find out not only what can be verified, but also  what is no longer valid in comparison to conventional module theory.

 \medskip


  \section{Tight Submodules}    

\medskip

The fundamental concept we are using throughout is the tightness of submodules in modules of p.d.$\le 1$.  The following definition applies to all rings. Let $B < A$ be modules such that  p.d.$A \le 1$. $B$ is called {\it tight in $A$} if p.d.$A/B \le 1$. Then also p.d.$B \le 1$.  It is evident that direct summands are tight submodules, but tight submodules need not be summands.  The tightness  for p.d.1 that was introduced in \cite{F} and immediately generalized to higher p.d.s  in \cite{BF}, was slightly different: its definition required that p.d.$A/B \le $ p.d.$A$. The only difference between the two variants is that in a free $V$-module {\cp ~ a free submodule} is not  tight in the sense of \cite{F}, but {\it is} in this paper.  (The present version was called {\it t-submodule} in \cite{F7}.)

The following easily verifiable basic rules will be applied, often without explicit reference:

\begin{lemma}\label{1} Let $C < B < A $ be $V$-modules and {\rm p.d.}$A \le 1$. \smallskip

\qquad {\rm (a)}  If $C$ is tight in $B$ and $B$ is tight in $A$, then $C$ is tight in $A$.

\qquad {\rm (b)}  If $C$ and $B$ are tight in $A$, then $C$ is tight in $B$.

\qquad {\rm (c)} If $C$ and $B$ are tight in $A$, then $B/C $ is tight in  $ A/C$.

\qquad {\rm (d)} If $C$ is tight in $A$ and  $B/C $ is tight in  $ A/C$, then $B$ is tight in $A$.
\qed
\end{lemma}

As already mentioned before, the symbol $\CC_V$ will denote the category of $V$-modules with the usual morphisms,  while $\CC_V^*$ is  the subclass whose objects are the $V$-modules of p.d.$ \le 1$.  E.g. the frequently used localizations $V_S$ at countable multiplicative submonoids $S$ of $V^\times$ are objects of our class $\CC_V^*$. The subobjects of an object $M \in \CC_V^*$ are the tight submodules of $M$. The morphisms in $\CC_V^*$ are those module homomorphisms $\f : M \to N $  ($M,N \in \CC_V^*$) for which Im $\f$ is  tight in  $N$. This also means that Ker $\f$ should be a tight submodule of $M$. Thus morphisms in $\CC_V^*$ have tight kernels and images.   Clearly,  a submodule $A$ of $M$ is tight if and only if the inclusion morphism $A \to M$ belongs to $\CC_V^*$. In order to check whether or not $\CC_V^*$ is a category, one ought to examine the axioms of category theory; in particular, the critical one that says that if $\a: A \to B$ and $\b: B \to C$ are morphisms, then so is the composite map $\b \a: A \to C$.  Unfortunately, this property is seldom satisfied in  $\CC_V^*$. As we will see in a moment, this property  is related to the one that the sum of two tight submodules is again tight -- which holds very rarely. In the next lemma this property is compared to similar properties, in particular to the tightness of the intersection of two tight submodules.

\begin{lemma} \label{sumcap} Let $M$ be a $V$-module of {\rm p.d.}$\le 1$, and $A,B$ two tight submodules of $M$. Consider the following conditions: \smallskip

{\rm (i)}   $A \cap B$ is tight in $A$;

{\rm (ii)} $A \cap B$ is tight in $B$;

{\rm (iii)}  $A$ is tight in $A+B$;

{\rm (iv)}  $ B$ is tight in $A+B$;

{\rm (v)} $A \cap B$ is tight in $M$;

{\rm (vi)} $A \cap B$ is tight in  $A+B$;

{\rm (vii)} $A + B$ is tight in $M$. \smallskip

\nin Then conditions  {\rm (i)-(v)} are equivalent,  {\rm (vi)} follows from each of them, and {\rm (vii)} implies each of them.
\end{lemma}

\begin{proof} (i) $\Leftrightarrow$ (iv) as well as (ii) $\Leftrightarrow$ (iii) follows from Noether's isomorphism theorem.

(i) $\Leftrightarrow$ (v).  The third  non-zero module in the exact sequence
$$0 \to A/(A\cap B)\to M/(A\cap B) \to M/A \to 0$$
has p.d.$\le 1.$ We deduce, by the well-known Kaplansky's lemma on p.d.'s in short exact sequences (see e.g. \cite[Lemma 2.4, p. 202]{FS2}) that the first two  non-zero modules in the last exact sequence are simultaneously of p.d. $\le 1$.

(vii) $\Rightarrow$ (ii).  From the exact sequence
$$0 \to (A+B)/A \to M/A \to M/(A+B) \to 0$$
where $M/A$ and $M/(A+B)$ have p.d.$\le 1$ we argue, by the same Kaplansky's lemma that p.d.$(A+B)/A=1$. Hence  p.d.$B/(A\cap B)=1$, and $A\cap B$ is tight in $B$.

 (ii) $\Rightarrow$ (v). By hypothesis $A\cap B$ is tight in $B$, and $B$ is tight in $M$.  Therefore, by Lemma 2.1(a), $A\cap B$ is tight in $M$. Similar argument yields  (i) $\Rightarrow$ (v).

 (vii) $\Rightarrow$ (iii). Because of  Lemma 2.1(b), $A$ and $A+B$ tight in $M$ implies $A$ is tight in $A+B$.

(i)+(ii) $\Rightarrow$ (vi). If  $A/(A\cap B)$ and $(A+B)/A  \cong
B/(A\cap B) $ have p.d.$\le 1$, then the  middle term in the exact  sequence $$0 \to A/(A\cap B) \to (A+B)/(A\cap B) \to (A+B)/A \to 0$$ is likewise of p.d.$\le 1$.
\end{proof}

We are indebted to L. Salce for furnishing us with an example showing that the sum of two tight submodules of a module of p.d.1 need not be tight even if their intersection is tight  (the converse is ruled out by the implication (vii) $\Rightarrow$  (v) in Lemma \ref{sumcap}).  Since the proof requires several results from the theory of \vds\ that are not needed in this paper, we skip the details.

\smallskip

 The relation between the module and its tight submodules is a fundamental issue.
 The following simple fact might provide useful information.

\begin{lemma} \label{gen}  Let $M$ be a $V$-module of {\rm p.d.}$\le 1$.   A tight submodule $C$ of $M$ satisfies $\gen C \le  \gen M.$ \end{lemma}

  \begin{proof}  If $\gen M = n$ is an integer, then  by Warfield \cite{W} $M$ is the direct sum of $n$ cyclic submodules, so its Goldie-dimension is $n$ (the {\it Goldie-dimension} | to be abbreviated as Gd | of a module $M$ is the supremum of the cardinalities of the sets of non-zero summands in direct sums $\oplus_{i \in I} M_i$ contained in $M$). A submodule $C$ cannot have larger Goldie-dimension, thus $\gen C \le n$.
 \[
        \begin{CD}
                 &&   &  &   0& & 0   \\
                      && & &         @VVV @ VVV \\
                     &&&& H  @= H \\
                      &&  &&        @VV{} V @ VV{}V \\
                   0      @>>>    C @>>>   N @>>>   F@>>>    0 \\
                    &&   ||   &&   @VV{}V            @VV{\psi}V  \\
                 0         @>>>   C  @>>>       M  @>{\f}> >  M/C  @>>>    0\\
                     && &&         @VVV @ VVV \\
                      &&&&     0 && 0
                             \end{CD}
  \]
  \indent    If  $\gen M = \k$ is an infinite cardinal, then consider a free resolution $0 \to H \to F \to M/C \to 0$ where $\gen F = \k$. The pullback $N$ of $\f: M \to M/C$ and $\psi: F \to M/C$ fits in the vertical exact sequence $0 \to H \to N \to M \to 0$. (See the commutative diagram with exact rows and columns.) Obviously, the free submodule $H$ of the free $F$ satisfies $\gen H \le \gen F =\k$, whence $\gen N \le \gen H + \gen M =\k$ and $\gen N \ge \gen M=\k$ follow. Thus $\gen N =\k$, and as $N \cong F \oplus C$, the inequality $\gen C \le \gen N = \k$ becomes evident.
\end{proof}

Next we verify a result on the tightness of modules in a chain.  Keep in mind that being tight is not an inductive property. $($Continuity in the next lemma means that $M_\r= \cup_{\s<\r } M_\s$
whenever $\r$ is a limit ordinal.$)$

  \begin{lemma} \label{chain} Suppose
$$   0 = M_0 < M_1 < \dots < M_\s < \dots < M_\t =M    \leqno (1)$$
is a continuous well-ordered ascending chain of submodules of the module $M$ with union $M$ such that  the modules $M_\s \ ( \s < \t )$ have {\rm p.d.}$1$.   If each module is tight in its immediate successor, then each module is tight in all of the successor modules of the chain,  and also in $M$. Furthermore,  {\rm p.d.}$M =1$.
\end{lemma}

\begin{proof} Apply Auslander's familiar criterion on the p.d.  of the union of a chain (see e.g. \cite[Chap. VI, Lemma 2.6]{FS2}) to the subchain  starting at $M_\s$,  to argue that p.d.$M_{\s+\r+1}/M_{\s+\r} \le 1$ for all ordinals $\r \ge 0$ implies that $M_\s$ is tight in every successor in the chain.
 That p.d.$M =1$ follows from the same lemma.  \end{proof}

 Let us agree that when we say ``object", we will mean   an object of the subclass $\CC_V^*$ of $\CC_V$   (remember: all objects have p.d.$\le 1$, and all subobjects are tight!). Manifestly,
 $\CC_V^*$  is closed under direct summands and direct sums; the projections onto direct summands and the embeddings of direct summands are morphisms in $\CC_V^*$. In the opposite direction,  we observe that neither the sum nor the intersection of two subobjects is  necessarily a subobject, and the union of an ascending chain of objects need not be an object.   The class $\CC_V^*$ is not additive in general; it {\it is} of course if  $V$ is a discrete \vd\ (DVD).


\medskip

\section{Fundamental Objects} \medskip        

Because of the  restrictive  composition rule of morphisms in our class, we have to learn from scratch if and how some of the familiar basic  concepts have to be modified to fit in. One should not be surprised if in some cases hard-to-believe situations have to be accepted.

\smallskip

{\bf Cyclically presented objects.} The simplest objects in our class  $\CC_V^*$  are the cyclically presented modules:  $V r/V s$   ($r , s \in V$) where $Vs  \le Vr$ are principal ideals. All cyclic submodules of $V r/V s$ are subobjects. Later on we will see that these are the only subobjects of $V r/V s \cong V(rs^{-1}) \le V$. Accordingly, only the principal ideals are subobjects in object $V$.

Multiplications by ring elements induce endomorphisms of $V r/V s$ whose kernels and images are subobjects. Morphisms  between cyclically presented modules are the same in $\CC_V^*$ as  in $\CC_V$,  therefore the  $V$- as well as the $\CC_V^*$-endomorphism ring of a \cp\ object is local.

\smallskip

{\bf Finitely generated objects.} The finitely generated objects behave very nicely in $\CC_V^*$. It is an elementary result that a finitely generated module over a valuation domain is of p.d.$\le 1$ if and only if it is finitely presented (see \cite[Proposition 4.1, p. 83]{FS1}), and then it is the direct sum of a finite number of cyclically presented modules (see  Warfield \cite{W},  or  \cite[p.159]{FS2}).
 Thus in our class  $\CC_V^*$, the finitely generated objects are  the finitely presented $V$-modules. Note that the annihilators of elements in such a module are principal ideals (the ideal 0 is included), thus subobjects.

  \begin{example} \label{\fp}  Let $\{x,y,z,w\}$ be a generating set of the module $M$ with defining relations  $rx =0,\ sy=0,\ ax+by +cz =0, ex +dw=0$ where $r,s, a, b,c,d,e$ are non-units in $V^\times$. The structure of  $M$ depends to a great extent on the divisibility relations between these ring elements.
 Assume e.g. the following proper divisibility relations: $c \mid d \mid b \mid e \mid a \mid s \mid r$.  We can choose a different generating set: $\{x,y,z',w' \}$ with $z' = z+ ac^{-1}x+ bc^{-1}y,\ w' = w+ ed^{-1}x$; then the new relations will be $rx =0,\ sy=0, \  cz'=0,\ dw'=0$.   This shows that the \fp\ module $M$ is the direct sum of four \cp\ submodules  generated by $x,y,z', w'$ with annihilators $Vr, Vs, Vc, Vd$, respectively.
 \qed \end{example}

If $B$ is a finitely generated submodule of a finitely presented module $A$, then $A/B$ is  finitely presented, so it has p.d.$\le 1$. Thus $B$ is a subobject in $A$. If $C$ is tight in a \fp\ module $A$, then p.d.$A/C \le 1$ implies that $A/C$ is \fp. Hence $C$ is \fg. Therefore, the subobjects in a \fg\ object are precisely the \fg\ submodules.
 An immediate consequence of this fact is  the following corollary. It is telling us that the subclass of \fp\ objects behaves the same manner in $\CC_V^*$ as in $\CC_V$.

\begin{corollary} \label{monoid} {\rm 1)} Let $\a : A \to B$ and $\b : B \to C$ be  morphisms between finitely presented objects in $ \CC_V^*$. Then the composite map $\b \a$ is also a morphism in $ \CC_V^*$.

 {\rm 2)} The endomorphisms  of  a \fp\ object in  $ \CC_V^*$ form a ring.
 \end{corollary}

  \begin{proof} It suffices  to show that the intersection of two \fg\ subobjects is also \fg\  (the sum of two \fg\ submodules is evidently again \fg). The implication (vii)$\Rightarrow $(vi) in Lemma \ref{sumcap} ensures that such an intersection is tight also in the sum of the objects. Tight in  \fg\ is \fg.
  \end{proof}

Important observations on \fg\ submodules in objects are recorded in the following proposition.

  \begin{proposition} \label{fp-sub} {\rm (i)}  Finitely generated submodules in any  object are  $($\fp$)$ subobjects. In a \fg\ object they are the only subobjects.

 {\rm (ii)}  The \fp\ subobjects of an object $M$ in the class $\CC_V^*$ form a sublattice in the lattice of submodules of $M$.

    {\rm (iii)} The sum and intersection of a \fp\  subobject with any subobject are also subobjects.
 \end{proposition}

\begin{proof} (i) This is an immediate consequence of \cite[Lemma 6.4, p. 217]{FS2}, which ensures that finitely generated submodules of modules of p.d.$\le 1$ are tight. The second part of claim (i) was already stated above.

(ii)  This follows from Corollary \ref{monoid}. 

   (iii) Let $A$ be a \fp\ and $B$ an arbitrary subobject of $M$.  The module $(A+B)/B$ is a \fg\ submodule of $M/B$,  therefore, it is tight in $M/B$. Hence $A+B$ is tight in $M$.  The rest follows from Lemma \ref{sumcap}.
\end{proof}

\smallskip

 {\bf Uniserial objects.}  Most important objects are the {\it uniserial} modules (also called serial modules in the literature): these are defined as modules whose submodules form a chain with respect to inclusion.  The obvious examples in $\CC_V$ are the so-called {\it standard uniserial modules}: the field of quotients, $Q$, as well as its submodules and submodules of their epic images, i.e. modules of the form $J/I$ where $0 \le I < J \le Q$ are submodules.  By Osofsky (see \cite{O} or also \cite[Chap. VI, Sect. 3]{FS2}), the p.d. of a submodule $J$ of $Q$ over a valuation domain is  an integer $n \ge 0$ if and only if gen $J = \al_{n-1}$ ($\al_{-1}$ means ``finite"). Hence uniserial objects in $\CC_V^*$   are at most countably generated. However,  | as noted before | a countably generated ideal $J$ is not
a subobject of $V$ in $\CC_V^*$; indeed, as p.d.$V = 0$ and p.d.$J= 1$ imply p.d.$V/J= 2$. A torsion uniserial ought to have p.d.$\le 1$ to belong to $\CC_V^*$, therefore only those standard uniserials $J/I$ are objects for which $J$ is at most countably generated and $I$ is cyclic, since p.d.$J/I \leq 1$ implies that $J/I$ is coherent; see e.g.  \cite[Chap. IV, Theorem 4.3]{FS1}. Thus all the annihilator ideals of elements in a uniserial object are
principal ideals. Hence it follows that the only proper subobjects of uniserial objects are
 cyclically presented. The non-standard uniserials (that play an important role in the theory of \vds, cf. \cite{BFS} or \cite{FS2})  will now be ignored, since their p.d. always exceeds 1.

A convenient way to deal with uniserial objects is to view them in the form $J$ (\tf\ case) or $J/V$ (torsion case), where $J$ is  either $Q$  (provided it is countably generated) or an at most \cg\  proper submodule of $Q$ containing $V$. Then it is trivial to answer the question of isomorphism of uniserial modules: $J \cong J'$  if and only if $J = rJ'$ or $J' = rJ $ for some $r \in V^\times$ (i.e. $J' = qJ$ for some $0 \ne q \in Q$), while $J/V \cong J'/V$ if and only if $J =J'$. Also, it makes sense to talk about total order even in the set of torsion uniserial objects in $\CC_V^*$: the order relation being induced by the natural inclusion relation  of the numerators $J$.

 Let $U$ be a uniserial object and $r \in V^\times$. It is obvious what is meant by $rU$. But it also makes sense to write
$r^{-1}U. $ Indeed, this denotes the  uniserial object $U'$ that satisfies $rU' = U$; it is unique up to isomorphism.

From the description of the proper  subobjects of a countably generated uniserial object  $U \in \CC_V^*$
 it follows that their only endomorphisms  are the automorphisms and the map to the zero submodule, that is, $\End_{\CC_V^*}(U) = \Aut_V(U) \cup \{ 0 \}$.
It is well known that the endomorphism ring of a uniserial module   in $\CC_V$ is a local ring (see Shores--Lewis \cite{SL}),  and  modules with local endomorphism rings enjoy the Exchange Property.

 The Exchange Property is one of the properties most frequently investigated about the behavior of summands.  Recall that a module  $A$ has the (finite) {\it Exchange Property} if direct decompositions $M = A \oplus B = C \oplus D$ of any module $M$ imply that there is another decomposition of the form $M = A \oplus C_1 \oplus D_1$ such that $C_1 \le C, D_1 \le D$.  Furthermore, $A$ has the {\it Cancellation Property} if $A \oplus B \cong A \oplus C$ for arbitrary modules $B,C$ implies $B \cong C$. Finally, $A$ enjoys the {\it Substitution Property} if $M = A_1 \oplus B = A_2 \oplus C$ with $A_1 \cong A_2 \cong A$ implies  the existence of a submodule $A'  \le M$ such that $A' \cong A$  and $M = A' \oplus B = A' \oplus C$. (See \cite{F0}.) We note that, if an object $A \in \CC_V^*$ has either the Exchange, or the Cancellation, or the Substitution Property
 as a module in $\CC_V$, then it also displays the same property in the class $ \CC_V^*$, since this class is closed under direct summands and direct sums.

In view of the preceding remarks, from \cite[Corollary 2.3, p. 342]{FS2} and \cite[p.181]{FS2} we derive the following theorem.

\begin{theorem} \label{end-u}  Finite direct sums of uniserial objects in the class  $\CC_V^*$ enjoy   all of the cancellation,  exchange and substitution properties. \qed
\end{theorem}

 {\bf Maximal uniserial submodules}. By a {\it maximal uniserial submodule in $M$} we mean a uniserial submodule that is not properly contained in another uniserial submodule of $M$.

 \begin{theorem} \label{max-us}  Suppose  $M$ is a  $V$-module of {\rm p.d.}$\le 1$, and   $U$ is a uniserial submodule in $M$. \smallskip

{\rm (i)} If $U$ is  maximal uniserial  in $M$, then it is at most countably generated and has {\rm p.d.}$\le 1$.

 {\rm (ii)} If $U$ is   \cg\ and tight in $M$, then it is a  maximal uniserial submodule in $M$.

  {\rm (iii)} If $U$ is a \cg\ maximal uniserial submodule in a tight submodule  of $M$, then it is also maximal in $M$.
\end{theorem}

\begin{proof} (i) See \cite[Chap. VI, Lemma 6.7]{FS2}.

(ii)   By way of contradiction assume $U$ is not maximal in $M$, i.e. there exists a uniserial $U' \le M$ that contains $U$ properly. We may assume that $U'$ is cyclic, say, $U' = Va$ for $a \in M$.  Then $Va/U$ is a non-zero cyclic submodule in the
module $M/U$ which has p.d.1 by hypothesis. Hence it follows that $Va/U$ is tight in $M/U$, so $U$ must be cyclic. This contradiction completes the proof of (ii).

(iii) The proof is similar to that of (ii). If $U$ is a maximal uniserial in a tight submodule $N$ of $M$ and contained in a larger uniserial $U' \le M$ that is cyclic, then $U'/U \cong (U'+N)/N$  is a cyclic submodule in $M/N$, so \cp. Again we can conclude that $U$ must be cyclic.
\end{proof}

   An immediate consequence of this theorem is that in a direct sum of \cp\ objects all  uniserial subobjects are also \cp.
\smallskip

{\bf Mixed modules as objects.}  An object in $\CC_V^*$ that is  mixed in the usual sense (i.e. neither torsion nor \tf)  need not have a 1-dimensional torsion part.  Actually,  the torsion submodule of a mixed module of p.d.1 can have any  p.d. not exceeding the maximal p.d. of \tf\ $V$-modules minus 1 whenever this number is $\ge 1$. Indeed, select a \tf\ $V$-module $N$ of p.d. $n \ge 2$ and set $N = F/G$ with a free $V$-module $F$. Let $H$ be an essential free submodule of $G$, and define $M = F/H$. Then $M $ is an object with torsion submodule $G/H$ of p.d. ~$n-1$.

 Even if the torsion submodule of a mixed object is an object, it need not be a subobject. Therefore, it seems reasonable to consider an object {\it mixed} in $\CC_V^*$ if its torsion submodule is a non-zero proper subobject. For an injective object in $\CC_V^*$ that is a mixed module in $\CC_V$, but not in $\CC_V^*$, see Example \ref{non-mixed} infra.

 \smallskip

{\bf Countably generated objects.}    Suppose $M$ is the union of a countable ascending  chain $M_n \ (n < \w)$ of \fp\ modules. Then each $M_n$ is  tight in its immediate successor and hence also in $M$ which will have p.d.$1$ (cf. Lemma \ref{chain}). Moreover, since every finitely generated submodule of $M$ is contained in some $M_n$, all finitely generated submodules of $M$ are tight in $M$ and finitely presented, so subobjects. Cyclic subobjects are cyclically presented, thus the annihilators of elements in $M$ are principal ideals (i.e. also objects).  Submodules that are countably generated may or may not be subobjects in countably generated objects. E.g. the countably generated ideals of $V$ are objects, but not subobjects of $V$.

 Claim (i) in our next theorem shows that the module $M$ in  the preceding paragraph is a typical \cg\ object.

  \begin{theorem} \label{cg}
{\rm (i)} A countably generated $V$-module is of {\rm p.d.}$\le 1$  if and only if it is the union
of a countable ascending chain of finitely presented $($tight$)$ submodules.

  {\rm (ii)} Let $A,B$ be tight submodules in a $V$-module $M$ of {\rm p.d.}$\le 1$ such that $B$ is at most \cg. Then \smallskip

\qquad 	{\rm (a)} $C =A \cap B $ is also at most \cg, and tight in $B$ and also in $A$ and in $M$;

\qquad 	 {\rm (b)} $A + B$ is  of {\rm p.d.}$\le 1$, and $A$ is tight in it.
\end{theorem}

  \begin{proof} (i) One way the claim follows from Proposition \ref{fp-sub}, while the converse is taken care of by the last but one paragraph before this theorem.

(ii) (a) $B$ is the union of a countable chain $\{B_n \ | \ n < \w\}$ of \fp\ submodules. Hence $B/C = B/(A\cap B)\cong (A+B)/A = \cup_n (A+B_n)/A $, where the last quotient modules  are \fg\ (epic images of the $B_n$ in $M/A$), and therefore tight in $M/A$.    The p.d. of their union $(A+B)/A$  equals 1, which means p.d.$B/C =1$. Thus $C$ is tight in $B$, and hence in $M$, and then also in $A$.

 (b)  The sum $A+B$ is the union of the chain of the extensions $A +B_n$ of   $A$ by \fp\ modules, so its p.d. is $\le 1$.
  \end{proof}

  Next we give more examples of countably generated objects. By choosing a suitable value group, it is easy to construct  \vds\ that have the required properties.

   \begin{example} \label{ex-cog}
Consider  a uniserial object $U$ generated by  $\{u_i \ (i <\w)\}$. Attach to each $u_i \in U$ a cyclically presented module, say,  generated by $b_i$ such that $r_ib_i =u_i $ where the non-units $r_i \in V^\times$ are required to satisfy the condition that $V/Vr_i$ is properly embeddable in $U/Vu_i$.    Then $U$ is a pure subobject of $B = \langle U, b_i \ (i< \w)\rangle$ such that $B/U$ is a direct sum of cyclically presented modules $\cong Vb_i/Vu_i$.  Since  pure extensions by a \cp\ module are obviously splitting, we obtain:
$B \cong U \oplus (\oplus_{i < \w} Vb_i/Vu_i)$.
\qed \end{example}

  \begin{example} \label{ex-uni}  Let $J, L$ denote submodules of $Q$ containing $V$ that are at most \cg, and let $r\in  V^\times$ be a non-unit. Then one of $J/Vr$ and $L/Vr$, say, the former, admits an isomorphic embedding in the latter, $\f: J/Vr \to L/Vr$, that is the identity on $V/Vr$.  Define $N$ as the push-out of the embeddings $\f: V/Vr \to J/Vr$ and $\psi: V/Vr \to L/Vr$. Then $N$ is an
 extension of its pure submodule $L/Vr$  by $J/V$. A fast calculation shows that
  $N =   L/Vr \oplus J'/V$ where $J'/V = \{(x, \f(x)) \ |\ x \in J/V\} \le J/V  \oplus L/V$.
 \qed \end{example}

{\bf The abundance of subobjects.} It is a remarkable fact that  even if $V$ is not a discrete  rank one \vd,  objects in $\CC_V^*$ contain a large number of tight submodules of all possible sizes. From our discussion above this is clear for finitely and countably generated objects, while for uncountably generated objects it is a consequence of the existence of {\it tight systems}.  Every module $ M$ of p.d.$\le  1$ admits a tight system $\TT$ (over any integral domain). This is defined as a $G(\al_0)$-family of tight submodules; see \cite[Chap. VI, Sect. 5]{FS2}. Recall that, for an infinite cardinal $\k$,  by a {\it $G(\k)$-family} $\GG$ of submodules of a module $M$ is meant a set of submodules such that the following conditions are satisfied: 1) $0, M \in \GG$; 2) $\GG$ is closed under unions; 3) if $X$ is a subset of $M$ of cardinality $\le \k$ and $A \in \GG$, then there exists a $B \in \GG$ such that $X \cup A \subseteq B$ and gen$(B/A) \le \k$.  It follows that in case $\TT$ is a tight system of $M$, then  $A < C \ (A,C \in \TT) $ implies $A$ is tight in $C$. Moreover, under the canonical map the  complete preimages  of the tight submodules in $C/A$ are subobjects of $C$; they are also subobjects in $M$.

An immediate corollary of the existence of tight systems is the next result.

  \begin{corollary} \label{embed} Let $M$ be an object, and $N$ a submodule of $M$. Then $N$ is contained in a tight submodule $N^*$ of $M$ such that $\gen N^* \le \max\{\gen N, \al_0\}$.
\end{corollary}

\begin{proof} Every generator of $N$ is contained in some \cg\ tight submodule that belongs to a fixed tight system $\TT$ of  $M$. The union of all these members of $\TT$ is a member $N^*$ of $\TT$ containing $N$. By construction,      $N^*$ can be generated by $\gen N \cdot \al_0$ elements.
\end{proof}

A tight system $\TT$ of object $M$ allows us to build a continuous well-ordered ascending chain (1)
of tight submodules $M_\s \in \TT$  for some ordinal $\t$, such that gen($M_{\s+1}/M_\s) \le \al_0$ for all $\s< \t$.  Moreover, since a \cg\ object is  the union of a chain of \fp\ subobjects (Theorem \ref{cg}), the chain (1) can be refined so as to have all of the quotients $M_{\s+1}/M_\s$ finitely, or even \cp.
\smallskip

Another important consequence of the existence of tight systems is that we have already formulated in Proposition \ref{fp-sub}: every \fg\ submodule of a module $M$     (of any size)  of {\rm p.d.}$1$ is \fp\ and tight in $M$.  \medskip

{\bf Projective objects.} A finitely generated torsion-free module over a valuation domain is free. A useful fact: a finite rank pure submodule in a free $V$-module is a summand; see e.g. \cite[Chap. XIV, Theorem 6.1]{FS1}.
By Kaplansky \cite{K},  projective $V$-modules are free. Evidently, they are projective objects in our class $\CC_V^*$ as well. Dimension calculation shows that a tight submodule of a free module must have zero p.d.   Therefore, we can state:

\begin{theorem} \label{subfree} The subobjects of free $V$-modules are the free submodules.  \qed
\end{theorem}

Hence we conclude  that every object $M$ in $\CC_V^*$ admits a free resolution in the form of a short exact sequence: $0 \to H \to F \to M \to 0$ where $H,F$ are free $V$-modules, i.e. free objects.

\medskip


\section{More Fundamental Concepts} \medskip   

 Continuing the review of the basics, we  would like to establish more results concerning the objects in the class  $\CC_V^*$, but in order to  deal with the objects more efficiently, we need several tools available in \cite{FS1} and in \cite{FS2}. In this section we review some concepts and facts we shall need.
\smallskip

{\bf Annihilators of elements.} The study of objects in $\CC_V^*$ is greatly simplified by the fact that the annihilator ideals of elements in objects are not just objects, but they are even principal ideals. This has been pointed  out before, but let us give a formal proof of this property.

\begin{lemma}\label{ann} Let $M$ be an object in $\CC_V^*$. Then for any element $a \in M$, the annihilator $\ann_M(a) = \{ r \in V \ | \ ra=0\}$ is a principal ideal of $V$.
\end{lemma}

\begin{proof}  $M$ has a tight system, so $a \in M$ is included in a \cg\ tight submodule $N$ of $M$, and hence also in a \fp\ submodule (see Theorem \ref{cg} above). For \fp\ objects the claim has been established before.
\end{proof}

\smallskip

{\bf Heights of elements.}  The principal information in describing the way an element is located in the module is stored in its height. Heights of elements are defined by using uniserial modules, see \cite{FS1}. The uniserials that occur as possible heights for  \vds\ have been studied in \cite{BF} and \cite{BFS}.  Fortunately, in modules of p.d.$\le 1$  only most tractable heights can occur.

Suppose  $M$ is an object and $0 \ne a \in M$. Consider maps $\f_J : J \to M$ of the
 submodules $J$ of $Q$ containing $V$ such that $\f_J (1) = a$. For a fixed $a$,  the union  in $Q$ of those $J$'s  for which such a  $\f_J$ exists is a submodule $H_M(a)$ of $Q$, called the {\it height-ideal}
of $a \in M$. The module $$h_M(a) = H_M(a)/V$$
is defined as the {\it height of $a \in M$}.  We call $h_M(a)$   {\it non-limit height} or {\it limit-height}
according as $H_M(a)$ is one of the $J$'s or is not.  In the  limit  case  we write $h_M(a)= U^-$.  Note that
$h_M(a)$ is always a uniserial torsion module; it is of the form $U = J/V$   with $J \subseteq Q$ (equality only in case $Q \in \CC_V^*$).  In the non-limit case, the element $a$ is contained in a uniserial module $W$ that is a maximal uniserial in $M$ such that $h_M(a) = W/Va$.  The heights of elements in a non-standard uniserial are uncountable limit heights | these are out of question in $\CC_V^*$. The set of heights occurring in $\CC_V^*$ is totally ordered in the obvious way once we declare the non-limit height $J/V$ to be
larger than the corresponding limit height  $(J/V)^-$.  The minimum height is 0 (this is the height of the generator in a cyclic module), and we set $h(0) = \infty$ as the maximum height.

\begin{example} \label{limit} To give an example of a limit height, consider  a \cg\ submodule $J$ of $Q$ containing $V$, and choose a properly ascending chain of fractional ideals $ \{Vt_i^{-1}\ | \ i  < \w\}$  with union $J$ (where $t_i \in V^\times$). Define a \cg\ object $X$ as follows: the generators are $x_i \ (i < \w)$ with the defining relations:
 $$rt_0x_0=0,\quad t_0 x_0 = t_ix_i \qquad (i <\w)  \leqno(2)$$
 where $r\in V^\times$ is arbitrary.  The element $t_0x_0$ has limit height, namely $(J/r^{-1}V)^{-}$. To get an idea of what kind of module $X$ is, observe that the cyclic submodules generated by the elements $x_i - (t_i^{-1}t_{i+1}) x_{i+1}$ are summands of $X$ for all $i < \w$ (a complement is the submodule generated by all the given generators with $x_i$ removed). Actually, these cyclic modules generate their direct sum $X'$ in $X$. This $X'$ is tight and pure in $X$, and the quotient $X/ X'$  is a \cg uniserial module containing the coset $x_0 + X'$.
\qed \end{example}

Next we prove the following:

\begin{theorem} \label{height} A non-zero element in an object of the class $\CC_V^*$ has one of the following heights: \smallskip

{\rm (i)} cyclic height;

{\rm (ii)} \cg\ non-limit height;

{\rm (iii)} arbitrary limit height  of standard type. \smallskip

Elements in a \fg\ module cannot have limit heights. \end{theorem}

\begin{proof} To begin with, observe that for each of (i)-(iii) we already had examples above, so it remains only to show that (i)-(iii) is  a complete list.  The only other heights in $\CC_V$ are uncountably generated non-limit heights.  Working toward contradiction, suppose that  for some $a \in M \in \CC_V^*$, we have $h_M(a) = J/V$ with an uncountably generated submodule $J$ of $Q$.  There is a homomorphism $\f_J : J \to M$ such that $\f_J (1) = a$. The maximal property of $J$ as height
implies that $\f_J(J) $ must be a maximal uniserial in  $M$. Therefore, by Theorem \ref{max-us} it is at most \cg\ | a  contradiction, completing the proof of the first claim.

  That (iii) cannot occur in a \fg\ module is an immediate consequence of the simple fact that limit heights require infinite Goldie-dimension, as is clear from  Example \ref{limit}.
\end{proof}

{\bf Height-gaps.} Suppose $U$ is a uniserial object and $Vr \ne 0$ is the annihilator of $a \in U$. Then $h_U(a) = rU$ and %
$h_U(sa) = s^{-1}rU$ provided that $sa \ne 0$ for $s \in V$.  If $U$ is contained in a $V$-module $M$, then the heights of these elements may be larger in $M$.  In general, in every module $M$, for an element $a\in M$ and its multiple $ra \ (r \in V^\times)$ the inequality
$$h_M(a) \le r^{-1} h_M(ra)$$
holds.
We say that $M$ has a {\it height-gap} at $0 \ne a \in M$   if, $h_M(a) > sh_M(x)$ holds whenever $sx =a $ for some $ x\in M$ and for a non-unit $s\in V$.

\begin{example}  To illustrate height-gaps, let $U$ be a uniserial module,  and  $x_1,x_2, x_3$ symbols.
 Suppose that the non-units $s_i, t_i \in V^\times $ satisfy the following proper divisibility relations: $ s_1 \mid s_2 \mid s_3 $ and $ t_1 \mid t_2 \mid t_3 $. Pick some $ u \in U$ such that $s_3u \ne 0$, and define a module $N$ to be generated by $U$ and by the given symbols subject to the relations
$ s_i u = s_it_i  x_i \ (i=1,2,3).$
The height-gaps in the submodule $U$ are at $s_1u, s_2 u, s_3u$, and at $0$.
\qed \end{example}

     {\bf Purity.}    The main point about this widely used concept that we are emphasizing repeatedly is that in \vds\ it is equivalent to the simpler relative divisibility (see \cite{W0}). Thus a   submodule $N $ is pure  in a $V$-module $M$ if and only if $rN = N \cap rM$ holds for every $r \in V^\times$.   Equivalently, for all $r \in V$, the map $V/Vr \otimes_V N \to V/Vr \otimes_V M$ induced by the inclusion $N \to M$ is monic.  This is tantamount to the injectivity of the  map $\Hom_V(V/Vr, M) \to \Hom_V(V/Vr, M/N)$ for all $r \in V^\times$ induced by the natural homomorphism $M \to M/N $.
 A {\it pure-exact sequence} $0 \to A \to B \to C \to 0$ is an exact sequence in which the image of the map $A \to B$ is pure in $B$.

 \begin{lemma}  {\rm (i)} Let $U$ be a uniserial submodule of an object $M$, and $a \in U$.  If $h_M(a) = U/Va$, then $U$ is a maximal uniserial in $M$, and there is no height-gap in $U$ at $a$ and above.

 {\rm (ii)} If $U$ is a maximal uniserial in $M$ and is torsion with no height-gaps other than the one at $0$, then $U$ is pure in $M$.
\end{lemma}

\begin{proof}  (i) This is rather obvious.

(ii) $U$ is not pure in $M$ means that there is $u \in U$ such that $h_U (u)< h_M(u).$
 Hence there must be a height-gap in $U$  at $u$ or above,  because by maximality, some of the generators of $U$ have the same height in $U$ as in $M$.
\end{proof}

We recall the definition of $\Pext_V^1(X,M)$:  it is a  sub-bifunctor  of $\Ext_V^1(X,M),$ consisting of those non-equivalant extensions of   $M$ by $X$ in which $M$ is a pure submodule (see e.g. \cite[p. 45]{FS2}). In the commutative case, Pext is a $V$-module.

\medskip


\section{Theorems on Torsion and Torsion-free Modules}  \medskip

In this section, we discuss briefly a few fundamental results on torsion and \tf\ objects. An in-depth study that would require more research and interesting applications is planned in the future.

We start with the following simple observation.

 \begin{theorem} \label{pure} A pure and tight \fg\ submodule in an object is a direct summand.
\end{theorem}

\begin{proof}  Suppose a \fg\ module $N$ is pure and tight in a module $M$ of p.d.$\le 1$. By the tightness of $N$, $M/N$ has p.d.$\le 1$, and by Corollary \ref{fppi} $N$ is pure-injective.  All this combined implies that $N$ is a summand of $M$.
\end{proof}

 We continue with
typical examples of direct sums of \cp\ modules: the {\it pure-projective} objects. These are defined as objects $P$ that satisfy $\Pext_V^1(P,M)=0$ for all objects $M  \in \CC_V^*$.

\begin{theorem} \label{pp} A $V$-module is pure-projective if and only if it is a direct sum of \cp\ modules.
\end{theorem}

\begin{proof} This is a special case of a well-known theorem. E.g. it follows from \cite[Chap. VI, Theorem 12.2]{FS2}.
\end{proof}

Concerning direct sums of uniserials, a most important result is the following theorem (this is not related to tightness).

\begin{theorem} \label{sumu} {\rm (i)} The uniserial summands in a direct sum of uniserial modules are unique up to isomorphism.

 {\rm (ii)}  Summands of a direct sum of uniserial modules are themselves direct sums of uniserials.
\end{theorem}

\begin{proof} These are well-known immediate consequences of the fact that the endomorphism rings of uniserial modules are local (Theorem \ref{end-u}).
\end{proof}

 Let  $r \in V^\times$ be  a non-unit.  By a {\it $V/Vr$-homogeneous} module we mean   a $V$-module $H$ such that each element  is contained in a submodule of $H$ that is $\cong V/Vr$. Then $H$ satisfies $rH=0$, and  any cyclic submodule of $H$  that is $\cong V/Vr$ must be pure in $H$.  Moreover, by Theorem \ref{pure} it is then a summand.

\begin{proposition} \label{homog} Suppose $M$ is a $V$-module of {\rm p.d.}$1$. \smallskip

{\rm (i)} If $rM=0$, then  a $V/Vr$-homogeneous tight submodule  is a summand of $M$.

{\rm (ii)} If $M$ is $V/Vr$-homogeneous, then it is the direct sum of \cp\ submodules, all isomorphic to $V/Vr$.

{\rm (iii)} If $D$ is a divisible object, then for every  $r \in V^\times$, $D[r]$ is $V/Vr$-homogeneous, so a direct sum of \cp\ submodules isomorphic to $V/Vr$.
\end{proposition}

 \begin{proof} For (i)-(ii) we refer to \cite[Chap. XII, Theorems 2.2 and 2.3]{FS2}, and for (iii) to \cite[Chap. XIV, Corollary 2.4]{FS2}.
 \end{proof}

 We note that the number of summands $\cong V/Vr$ in (iii) is the same for every $r$ provided that $D[r] \ne 0$: it is the Goldie-dimension of $D$.

  \smallskip

   An important theorem in abelian group theory, due to L. Ya. Kulikov, states that a subgroup of a direct sum of cyclic groups is likewise a direct sum of cyclic groups (see, e.g., \cite{F1}). An analogue in $\CC_V^*$ would state that a tight submodule of
 a direct sum of cyclically presented modules is also such a direct sum.  This is indeed true for \tf\ modules: a tight submodule of a free module in $\CC_V^*$ is again free.  It was conjectured that this holds in $\CC_V^*$   also in the torsion case.  (For torsion abelian groups, see  e.g. \cite[Chap. 3, Theorem 5.7]{F1}.) However, we claim that the module $X$ in Example \ref{limit}  refutes this conjecture. In order to prove this, consider the module $Y$ in the following example.

  \begin{example} \label{ex1}  The \cg\ torsion object $Y$  is defined just as the module $X$ in Example \ref{limit}: it is generated by the same set $\{ x_i \ |\ i < \w\}$ with the same defining relations, but there is a single modification:  we replace $t_0 \in V^\times$ by $s \in V^\times$ that is picked such that $J < Vs^{-1}$. In this case, $Vx_0$ is a pure and tight submodule in $Y$ (a summand), and the elements $x_i-(t_i^{-1}s)x_0$ for all $i >0$ generate cyclic direct summands of $Y$ such that $Y$ is the direct sum of $Vx_0$ and these cyclic submodules.    (Another, but less explicit argument to obtain the structure of $Y$ is as follows.  After observing that the cyclic submodule $Vx_0$ is pure in $Y$, it only remains to point out that moreover, it is a summand of $Y$, since $Y/Vx_0$ is pure-projective as the direct sum of \cp\ modules $\cong Vt_i \ (i >0)$.)
   \qed \end{example}

 To argue that the object $X$ of Example \ref{limit} cannot be a direct sum of  \cp\ objects,  appeal to Theorem \ref{height}.  The element $t_0x_0 \in X$ is of countable limit height, and as such it  cannot belong to a direct sum of the stated  kind: it would be contained already in  a \fg\ summand with the same limit height. However, this is impossible  as is demonstrated by the cited theorem. Thus the object $X$ that is (isomorphic to) a tight submodule (observe  that $Y/X \cong Vs/Vt_0$ is \cp)  in a direct sum $Y$ of \cp\ modules fails to be a direct sum of such modules.

 Hence it is obvious  that this theorem of Kulikov  cannot have the  suspected analogue in $\CC_V^*$  without additional hypotheses. Looking for simple conditions that would  lead  us to a Kulikov-type theorem  for subobjects in direct sums of  \cp\ objects, we selected (b), in addition to the obvious (a), that seems  natural  to assume:  \smallskip

  {\rm  (a)} {\it The non-zero elements have cyclic heights. }

{\rm  (b)} {\it The  uniserial submodules   admit but a finite number of  height-gaps.} \smallskip

\nin Under the hypothesis of (a)-(b), we will prove a desired analogue for the \cg\ torsion objects (see Theorem \ref{kulikov} below). But first we deal with preliminary lemmas.

We need a definition. Similarly to \cite{FS0}, we will call a $V$-module $M$  {\it cyclically separable} if every finite set of its elements can be embedded in a \fg\ summand of $M$, i.e.  in a summand that is the direct sum of a finite number of \cp\ modules.   Observe that in order  to verify the cyclic separability of a torsion object, it suffices to check the defining property only for one element subsets, as every \fg\ object is a finite direct sum of \cp\ objects. Hence it is evident  that summands of modules inherit cyclical separability. We now prove a crucial lemma.

\begin{lemma} \label{sep} Let $M$ denote a torsion object in $\CC_V^*$.  If $M$ satisfies conditions {\rm (a)-(b)}, then it is a cyclically separable $V$-module. \smallskip
\end{lemma}

\begin{proof}
Assume $M$ has properties (a)-(b), and let $0 \ne a \in M$. By (a), $a$ is contained in a \cp\ submodule $C= Vc$ that is maximal uniserial in $M$.
If $C$ contains no height-gap strictly between $a$ and $0$, then $C$ is pure in $M$, and hence a summand of $M$ (Theorem  \ref{pure}). Thus in this case $a$ embeds in a \cp\ summand of $M$, and we are done.
If there are height-gaps in $C$  between $a$ and $0$, then by (b) there is one, say at $rc \ (r \in V)$, such that no height-gap exists strictly between $rc$ and $0$. Then by the previous argument there is a \cp\ summand $B=Vb$ of $M$ that contains $rc$, $M = Vb \oplus M'$.  If $b' \in Vb$ is such that $Vrb'/Vc \cong Vr$, then $V(c)+V(b) = V(c-b') \oplus V( b)$. In this case, the projection of $V(c-b')$ in $M'$ contains the coordinate of $a$ with a smaller number of height-gaps below it. Repeating this process for the coordinates of  $a$ a finite number of times, we get a \fg\ summand of $M$ that contains the selected element $a$.
\end{proof}

 \begin{lemma} \label{Kuli} Let $M$ be a torsion object satisfying conditions {\rm (a)-(b)}. If $M$ is \cg, then it is a direct sum of \cp\ objects.
 \end{lemma}

 \begin{proof} By Lemma \ref{sep}, $M$ is cyclically separable. It is a simple exercise to prove that   a \cg\  cyclically separable module is a direct sum of cyclics.
 \end{proof}

The following analogue  of Kulikov's theorem is now easily established.

 \begin{theorem}   \label{kulikov} Assume $M$ is a direct sum of  \cp\ torsion objects, and $N$ is a \cg\ subobject  satisfying condition {\rm (b)} above.  Then $N$ is likewise a direct sum of   \cp\  subobjects.
\end{theorem}

 \begin{proof}
Owing to  Theorem \ref{gen} and Lemma \ref{sep} it suffices to show that $N$ satisfies condition (a). But this is immediate by virtue of Theorem \ref{max-us} (iii).  \end{proof}

 We are asking  the obvious question: do the preceding lemma and theorem hold for larger cardinalities? The answer is:  for Lemma \ref{Kuli} counterexamples are torsion-complete abelian $p$-groups with countable unbounded  basic subgroups; cf. \cite[Chap. 10, sect. 3]{F1}. For Theorem \ref{kulikov} we do not know the answer.
 
We record the following two parallel questions.

\begin{problem} Are pure subobjects in direct sums of torsion uniserial $($resp. \cg$)$ objects also direct sums  of the same kind?
\qed \end{problem}

\medskip

Next we want  to get an idea of the \tf\ modules in $\CC_V^*$. It is a pleasant surprise that all countably generated \tf\ modules in $\CC_V$ are objects in $\CC_V^*$.  This is evident from the following theorem.

\begin{theorem} \label{tf1}  {\rm (i)}  A \tf\ $V$-module $A$ has {\rm p.d.}$\le 1$  if and only if every rank one pure submodule is at most \cg.

  {\rm (ii)} A torsion-free $V$-module $A$ is of {\rm p.d.}$\le 1$ if and only if it admits a well-ordered ascending chain of tight pure submodules $A_\a$ such that  for each  $\a$, $A_{\a+1}/A_\a$ is of rank  one and of {\rm p.d.}$1\ ($thus cyclic or countably generated \tf$)$.
\end{theorem}

  \begin{proof}  It suffices to refer to \cite [Corollary 4.5]{F2} and to \cite[Chap. VI, Lemma 6.6] {FS2}, respectively.
\end{proof}

We continue with a theorem that resembles Pontryagin's theorem on countable free abelian groups.  A similar result for the projective dimension one case  is also  included.

\begin{theorem} \label{Pontr} A \tf\ module of countable rank in $ \CC_V^*$ is free $($is an object in $\CC_V^*)$ if and only if its finite rank pure submodules are free $($have {\rm p.d.}$\le 1)$.
\end{theorem}

  \begin{proof} See \cite[Chap. VI, Corollary 3.12]{FS2}.
\end{proof}

\medskip


\section{Divisible and Injective Objects} \medskip   

  The theory of divisibility and injectivity clearly illustrates a fundamental difference between the classes $\CC_V^*$ and $\CC_V$. \smallskip

{\bf Divisible objects.}  Divisibility of modules is defined as usual: $D \in \CC_V^*$ is {\it divisible} if $rD=D$ for all $r \in V^\times$. Equivalently,  the equality $\Ext^1_V(V/Vr,D)=0$ holds for all $r \in V^\times$.  The prototype of divisible modules,  the quotient field $Q$ of $V$ as a $V$-module,  is in general not an object. It {\it is} exactly when $Q$ is a countably generated $V$-module (then p.d.$Q=1$, i.e. $V$ is a {\it Matlis domain}). But the module $\partial_V$ (see \cite{F}), the generator of the subcategory of the divisible modules in $V$-Mod has p.d.1, so it is an object in $\CC_V^*$. Recall that $\partial_V$ is generated by the $k$-tuples $(r_1, \dots, r_k) $\ for all $k \ge 0$ of   non-unit elements $r_i \in V^\times$, subject to the defining relations
$$  r_k ( r_1, \dots, r_{k-1}, r_k) = ( r_1, \dots, r_{k-1})   \qquad (k > 0)  \leqno (3)$$
for all choices of the $r_i$. The generator $w =(\emptyset)$ generates a submodule of $\partial_V$ isomorphic to $V$ such that $\partial_{V_0}=\partial_V/ Vw$ is a divisible torsion module of p.d.1 (which is a generator of the subcategory of divisible torsion modules in $\CC_V$).  See \cite{F} or \cite{FS2} for more details.

  As far as  the structure of divisible objects is concerned, the following information is crucial. (Pay attention to the enormous simplification over Matlis domains.)

\begin{theorem} \label{div-str} {\rm (i)}  An object in $\CC_V^*$ is divisible if and only if it is a summand of a direct sum of copies of the module $\partial_V$.

  {\rm (ii)} If $V$ is a Matlis domain, then an object is  divisible if and only if it is the direct sum of copies of $Q$ and/or $K$.
\end{theorem}

\begin{proof} (i)  In \cite[Theorem 18]{F} it is shown that over a Pr\"ufer domain (and hence over a \vd) a divisible module has p.d.1 if and only if it is a summand of a direct sum of copies of $\partial$.  (By the way, this holds for all integral domains.)

(ii) See \cite [Chap. VII, Theorem 3.5]{FS2}.
\end{proof}

 In order to obtain a full set of invariants for a divisible object $D$, we introduce two cardinal invariants    measuring the size of its torsion and \tf\ parts. One is $ \k =\rk D$, the \tf\ rank of $D$, the number of generators of  a maximal size free submodule contained in $D$. The other invariant is  $
  \l = \gen D[r] $ for any non-unit $r \in V^\times$. Thus $\l$ is the cardinality of the set of summands $\cong V/Vr$ in a direct decomposition of $D[r]$ into  indecomposable summands.  These two cardinals form a complete set of invariants characterizing divisible objects in $\CC_V^*$. In fact,

\begin{theorem} \label{isodiv}  Assume $D$ and $D'$ are divisible objects in $\CC_V^*$.  Then $D \cong D'$ if and only if  \smallskip

{\rm (i)} their ranks are equal: $\rk D = \rk D'$; and

 {\rm (iI)}  for some, and hence for each $r \in V^\times$, $\gen D[r] = \gen D'[r] $.
 \end{theorem}

\begin{proof}  See  \cite[Theorem C]{F8} or \cite[Chap. VII,  Theorem 3.4]{FS2}.
\end{proof}

We also state the existence theorem accompanying this structure theorem.

\begin{theorem} \label{exdiv} Given the cardinals $\k, \l$, there exists a divisible object $D$ in class $\CC_V^*$ such that $\rk D = \k$ and $\gen D[r] = \l$ if and only if   \smallskip

 {\rm (i)}  in case {\rm p.d.}$Q =1$:  both $\k$ and $\l$ are arbitrary;

 {\rm (ii)} in case {\rm p.d.}$Q >1$:  $\k$ is arbitrary and $\l \ge \max\{\k, \gen Q\}.$

  \end{theorem}

\begin{proof} We refer to  \cite[Theorem 3]{F8} or to \cite[Chap. VII,  Theorem 3.8]{FS2}.
\end{proof}

 From the foregoing results we can draw the conclusion  that in case p.d.$Q  > 1$  every divisible $D \ne 0$ in $\CC_V^*$ satisfies Gd($D) \ge \gen Q$. Furthermore,  no  indecomposable divisible object exists in $\CC_V^*$.

\smallskip
We also have the embedding result as expected:

\begin{theorem} \label{div}  Every object in  $\CC_V^*$ is a subobject of a divisible object.
\end{theorem}

\begin{proof} Write $M \in \CC_V^*$ as $ M = F/H$ with free $V$-modules $F,H$. If $F = \oplus_{i \in I} Vx_i$ with $Vx_i \cong V$, then define $G =  \oplus_{i \in I} \partial_i$ with $\partial_i \cong \partial_V$, and embed $F$ in $G$ by identifying the generator $x_i$ with the generator $w_i \in \partial_i$, for all $i$. Then $F$ becomes a subobject of $G$, and $G / H$ will be a divisible module of p.d.1 that contains a copy of $M$ as a subobject.
\end{proof}

\medskip

$h$-divisibility of a $V$-module $H$ is defined by the extendibility of the homomorphisms $V \to H$ to $Q \to H$  (see \cite{M0} or \cite[p. 38]{FS2}). With the exception of the next proposition, this concept will not be discussed, considering that  $h$-divisible modules rarely exist in $\CC_V^*$, even injective objects are not $h$-divisible whenever p.d.$Q >1 $.

 \begin{proposition} \label{h-div} {\rm (i)} $h$-divisible objects exist in $\CC_V^*$  if and only if    $V$ is a Matlis domain, in which case all divisible modules are $h$-divisible.

{\rm (ii)}  If $V$ is a Matlis domain, then an  $h$-divisible object of {\rm p.d.}$1$ is the direct sum of copies of $Q$ and $K$.
\end{proposition}

 \begin{proof}  (i) It is well known that over a domain, divisibility and $h$-divisibility are equivalent if and only if p.d.$Q  \le 1$ (see e.g. \cite[Chap. VII, Theorem 2.8]{FS2}).

(ii)  See \cite[Chap. VII, Sect. 2]{FS2}.
  \end{proof}

\medskip

{\bf Injective objects.}
The role of injective modules  is played by objects $E \in \CC_V^*$ that satisfy $\Ext_V^1(A,E) =0$ for all $A \in  \CC_V^*$; i.e.   whenever $E$ is a subobject,   it must be a summand. Luckily, this property is equivalent to the more familiar extensibility of morphisms  into $E$  from subobjects to objects. But this equivalence comes with a caveat: the extended map need not be a $\CC_V^*$-morphism, since its image might not be tight. Perhaps unexpectedly, injectivity and divisibility turn out to be equivalent.

\begin{theorem} \label{inj-div} The following conditions are equivalent for an object $E \in \CC_V^*$. \smallskip

{\rm (i)} $\Ext_V^1(C,E) =0$ for all $C \in  \CC_V^*$;

 {\rm (ii)}  every morphism $\f: A \to E$ in $\CC_V^*$ extends to a  homomorphism $\psi: B \to E$   whenever $A,B$  are in $\CC_V^*$ and $ A$ is tight in $B$;

{\rm (iii)} $E$ is a divisible object.
\end{theorem}

\begin{proof}
In the category $\CC_V$ we have an exact sequence
$$\Hom_V(B,E) \to\Hom_V(A,E)\to  \Ext_V^1(B/A,E) \to \dots  \leqno(4)$$

(i) $\Rightarrow$ (ii). Hypothesis implies that Ext in (4) vanishes, so the map between the two Homs is surjective.

(ii) $\Rightarrow$ (iii). Condition (ii) ensures that for every $r \in V^\times$, every map $\f: Vr \to E$ extends to $V \to E$ (note that $\f \in \CC_V^*$).  This is equivalent to the divisibility of $E$ by $r$.

(iii) $\Rightarrow$ (i). By Bazzoni--Herbera \cite{BH}, over an integral domain $R$, a module $E$ is divisible if (and only if) $\Ext_R^1(C,E) = 0$ holds for all $R$-modules $C$ of p.d.$\le 1$.  (This implication was proved in \cite[Theorem 6]{F7} for Pr\"ufer domains.)
\end{proof}

As an immediate corollary to the preceding theorem we obtain:

\begin{corollary} Direct sums of injective objects are likewise injective. In particular,  injective objects are $\Sigma$-injective, i.e. any direct sum of  copies of an injective object is injective.
\qed \end{corollary}

A well-known test for the injectivity of a module is that its extensions by cyclic modules are splitting. In $\CC_V^*$ this criterion simplifies to \cp\ modules:

\begin{theorem} \label{test} An object $E \in \CC_V^*$ is injective if and only if $\Ext_V^1(C, E)=0$ holds for all \cp\ objects $C$.
\qed \end{theorem}

 It is well known in commutative module theory  that every module contains a unique maximal divisible submodule. This is not true in $\CC_V^*$ in general, because the relevant property that the sum of two divisible objects is again one fails; indeed, the property of being of p.d. at most $1$ is frequently lost when forming the sum.

\begin{example} \label{non-mixed} We exhibit  an injective object  which is a mixed $V$-module, but neither torsion, nor torsion-free, nor mixed as an object of $\CC_V^*$.    Let $V$ be a valuation domain such that p.d.$Q=3$, and consider the divisible $V$-module $\partial_V$ defined above. As p.d.$\partial_V=1$, we have $\partial_V \in \CC_V^*$.   The torsion submodule $T$ of $\partial_V$ has p.d.2, since $\partial_V/T \cong Q$. Therefore $\partial_V$, as an object of $\CC_V^*$, is neither torsion, nor torsion-free, nor mixed.
\qed \end{example}

The following corollary is obvious in view of our discussion of the divisible objects.  (For the following (i), cf. Theorem \ref{div}.)

\begin{corollary} \label{inj-res}
{\rm (i)} Every object embeds as a subobject in an injective object.

{\rm (ii)} Objects that are epic images of injective objects  $($modulo subobjects$)$ are themselves injective.

{\rm (iii)} Every object $M$ admits an injective resolution, that is an exact sequence $0 \to M \to A  \to B \to 0$ of modules of {\rm p.d.}$\le 1$ where $A,B$ are injective objects.

{\rm (iv)} The injective dimension of any object in $\CC_V^*$ is $0$ or $1$.
\qed  \end{corollary}

 Let us pause for a moment to answer a question concerning the existence of injective envelopes in the class $\CC_V^*$.   By the {\it injective envelope}  of an object $M$ we mean an injective object $E(M)$ containing $M$  as a subobject such that for every injective object $E$ containing $M$ as a subobject, the identity map of $M$ extends to a tight embedding $E(M) \to E$.  Of course, if an envelope exists, it is then   unique up to isomorphism.

 \begin{theorem} All modules in the class  $\CC_V^*$ admit injective $($divisible$)$ envelopes if and only if  $V$ is a  rank one discrete \vd.
\end{theorem}

\begin{proof} If $V$ is a DVD, then  $\CC_V$ and $\CC_V^*$ are identical, and the claim in  $\CC_V$ is well-known.  If $V$ is a Matlis domain, then all divisible modules are $h$-divisible,  they are direct sums of copies of $Q$ and/or $K$. If $V$ is not a DVD, then it contains a countably generated ideal, and  such an ideal cannot be tight in a direct sum of $Q$s.
Finally, if p.d.$Q >1$, then the $\CC_V^*$-injective envelope of $V$ should be an indecomposable summand of $\partial_V$, but |  as observed above | such a summand does not exist.
\end{proof}

\medskip

 \section{ Pure-Injectivity} \label{pi} \medskip

Pure-injectivity in $\CC_V^*$ can be  developed like in traditional module theory (see e.g. \cite[Chap. XI, Section 2]{FS1}  and \cite[Chap. XIII, Sections 2-3]{FS2}). As expected, there are several changes, so we provide  the whole proof, but skip routine arguments.

By a {\it system of equations}  (with unkowns $x_j \ (j \in J)$) over a module $M$ we mean a system of linear  equations
$$\sum_{j \in J_i}r_{ij} x_j = a_i \in M \qquad (r_{ij} \in V, \  i \in I)  \leqno(5)$$
for $i \in I,\  j \in J$ where $I,J$ are arbitrary index sets, and $J_i $ is a finite subset of $J$ for each $i \in I$.   Let $F$ denote the free module generated by the unknowns $x_j \ (j \in J)$  and $H$  its submodule generated  by the left sides  of the equations for all $i \in I$. We consider only {\it consistent} systems; i.e., systems that do not contain hidden contradiction. This means that we get a genuine homomorphism $\f: H \to M$ by mapping the generators of $H$ onto the elements of $M$ as shown by the equations, i.e. $\f: \sum_{j \in J_i}r_{ij} x_j \mapsto a_i $. It is easy to check that  the system has a solution in $M$ if and only if $\f$ extends to a homomorphism $\f^*: F \to M$, in which case $x_i = \f^*(x_i) \in M$ are the solutions.  It is pretty obvious that a consistent system defines an extension $M^*$ of $M$ by $F/H$ by adjoining to $M$ the unknowns $x_i$ as generators subject to the defining relations (5). Furthermore, it is straightforward to check that $M$ will be pure in $M^*$ exactly if (5) is finitely solvable in $M$, i.e. the finite subsystems of (5) admit solutions in $M$.

 We   define  p.d.($F/H$) as the
 {\it projective dimension} of the system (5). It will be convenient to call  (5) {\it  an adequate  equation system}  if  1) it is   consistent;  2)  its p.d.   is $\le 1$;  and 3)  it is finitely solvable.

An object  $M \in \CC_V^*$ is said to be  {\it pure-injective} if it has either one of the  equivalent  properties listed in the following theorem.  

\smallskip

 \begin{theorem} \label{pi} For an object $M$, $(\a)$-$(\g)$ are equivalent properties: \smallskip

 $(\a)$ $\Pext_V^1(C,M) = 0$ for all objects $C$.

$(\b)$ If $A$ is a pure subobject of object $B$,   then every $\CC_V^*$-map  $A \to M$ extends to a homomorphism $B \to M$   $($that need not be a $\CC_V^*$-map$)$.

 $(\g)$ Every  adequate   equation system over $M$ has a global solution in $M$.
\end{theorem}

\begin{proof} All modules   in this proof are objects in $\CC_V^*$.

($\a) \Rightarrow (\b$) Assuming $(\a)$,  consider the following push-out diagram   where  the top sequence is pure-exact  and $\zeta$ is a $\CC_V^*$-morphism.
 \[
        \begin{CD}
            0      @>>>    A    @>>>    B @>>>    C   @>>>       0 \\
                             &&   @VV\zeta V   @VVV  @|        \\
                       0 @>>>    M    @>>>    N@>>>    C  @>>>        0\\     \end{CD}
\]
Then the bottom row is also pure-exact, so it splits by hypothesis. Hence there is a homomorphism $B \to M$ making the upper triangle $ABM$ commute. This is an extension  of $\zeta$, proving $(\b)$.

($\b) \Rightarrow (\g$) Given an adequate    equation system (5) over $M$, consider the corresponding free module $F$ and  its  submodule $H$.  If (5) is viewed as a system over $M^*$, then by construction, there is an extension $\psi: F \to M^*$ of $\f: H \to M \le M^*$. This means that (5) is solvable in $M^*$. Hypothesis $(\b)$ implies that $M$ is a summand of $M^*  $. Hence $\psi$ followed by the projection $M^* \to M$ yields a desired extension $F \to M$ of $\f$.

($\g) \Rightarrow (\a$)   In the next diagram, let the bottom row  represent a pure extension of $M$ by $C$, and the top row a free resolution of $C$.  The map $\f^*$ exists because $F$ is free (making the right square commute).  It is evident that its restriction $\f$  to $H$ makes   the left square commute.   As $H$ is tight in $F$, the pair $\{H,F\}$ along with $\f$ defines a  system (5) of equations that is finitely solvable  in $M$, due the purity of the bottom sequence. Thus (5) is an adequate  system, and hence condition $(\g)$ implies that there exists a map $F \to M$ that makes the maps in the upper triangle $HFM$ commute. Then the bottom sequence splits,  establishing ($\a$).
 \[
        \begin{CD}
            0      @>>>    H   @>>>    F @>>>    C   @>>>       0 \\
                             &&   @V\f VV   @VV\f^* V  @|        \\
                   0       @>>>   M @>>>    N @>>>   C    @>>>        0\\
                         \end{CD}
\]  
\end{proof}

Evidently, divisible (i.e. injective) objects are pure-injective. Moreover, they contain a lot of pure-injective subobjects | as is shown by the following theorem.

 \begin{theorem} \label{Div[-r} Let $D$ be a divisible   object in $\CC_V^*$.  For every  $r \in V^\times$, the submodule $D[r] = \{d \in D\ | \ rd=0\}$ is a pure-injective object. \end{theorem}

\begin{proof}  From  the isomorphism $ D/ D[r] \cong D$ and from p.d.$D =1$ we infer that $D[r]$ is tight in $D$.
 Let $A$ be a pure subobject of object $B$, and $\xi: A \to D[r]$ a $\CC_V^*$-map. $\xi$ induces a map $\xi': A/rA \to D[r]$ that extends (by purity) to $\xi'': B/rB \to D$. Evidently, $\Im \xi'' \le D[r]$ as well, so the canonical map $B \to B/rB$ followed by $\xi''$ yields a desired extension $B \to D[r]$ of $\xi$.
 \end{proof}

Imitating the proof of Eklof--Mekler \cite[Chap. V, Corollary 1.3]{EM}, we verify:

\begin{theorem} \label{ppi} Every object in $\CC_V^*$ is a pure subobject in a pure-injective object. \end{theorem}

 \begin{proof}  Select any cardinal $\k > \max\{|V|, \al_0\}$. Given a module $M \in \CC_V^*$, we define a continuous well-ordered ascending chain $\{M_\s \ |\ \s < \k\}$ of length $\k$ as follows. Start with $M_0 =M$. If for some $\s$ the modules $M_\r$ of p.d.$\le 1$  have been constructed for all $\r \le \s$, then define $M_{\s+1}$  by adjoining to $M_\s$ the unknowns (as additional generators) of every  adequate   equation system   with defining relations given by the systems. It is readily checked that then p.d.$M_{\s+1} \le 1$ as well,  and    $M_\s$ will be tight and pure in $M_{\s+1}$ such that all the  adequate   equation systems    over $M_\s$  are solvable in $M_{\s+1}$.  At limit ordinals, we take the union which will again have p.d.$\le 1$ and will contain all previously constructed $M_\r$s as tight pure submodules. It is straightforward to check that the union  of the constructed chain will satisfy condition $(\g)$, and thus it will be a pure-injective object containing $M$ as a pure subobject.
(The process can stop at systems with   $\l$   unknowns, where $\l$ is any uncountable cardinal $> | V|.)$
\end{proof}

The next proposition  implies that the first Ulm-submodule of a pure-injective object is an injective object whenever its p.d. is $\le 1$.

\begin{proposition} \label{Ulm} The first Ulm-submodule $M^1 = \cap_{r \in V^\times} rM$ of a  pure-injective object $M$ satisfies $\Ext_V^1(C, M^1)=0$ for all objects $C$.
\end{proposition}

\begin{proof}  Let $\{r_i \ | \ i \in I\}$ be a list of the non-unit elements of $V^{\times}$.  We have to show that for any given  $a \in M^1$,  for each $r_j$ the equation $r_j x =a$ is solvable for $x \in M^1$. For each $j \in I$, consider the following system of linear equations
$$ r_j x = a, \quad r_ix_i=x \quad  (i \in I).$$
  An easy calculation confirms  that the p.d. of this system is $1$. Furthermore, since the chosen element $a$ is divisible by $r_jr_i$ for all indices, each system is finitely solvable in $M.$  By hypothesis, it has a solution  in $ M$. Clearly, each solution $x$ belongs to $ M^1$, so  $M^1$ is a divisible submodule.
\end{proof}

 For every \vd\ $V$ of global dimension   $\ge 2$, we exhibit an example of a pure-injective object in $\CC_V^*$ whose injective submodule is an object, but neither a subobject nor a summand  in $\CC_V^*$.

 \begin{example} \label{sd} Let $V$ be as stated.
We form the following  diagram with pure-exact rows and  commutative squares.
 \[
        \begin{CD}
            0      @>>>    A    @>>>    B @>>>    C   @>>>       0 \\
                             &&   @VVV   @VVV  @|        \\
                   0       @>>>   D @>>>    B' @>>>   C    @>>>        0\\
                     &&   @|    @VVV @VVV  \\
                       0 @>>>    D    @>>>    B'' @>>>    C'  @>>>        0\\     \end{CD}
\]
\smallskip
\nin For the top row  select a pure-exact sequence of torsion $V$-modules such that $A, B$ are of p.d.1, and p.d.$C=2$.   We can make the selection such that the first Ulm-submodule of $C$ is $0$. Embed $A$ in a divisible object $D$   as a tight submodule, and get the middle row  as a pure-exact sequence via pushout. The next step is the application of
 the embedding process of Theorem \ref{ppi} to $C$  (it works even if $C$ is not an object) to obtain a pure extension $C'$ by a module  $H$ of p.d.$1$ such that $C'$ has the property that its pure extensions by $V$-modules of p.d.$\le 1$ are splitting.  Since p.d.$H =1$, there is a module $B''$ making   (though not uniquely) the bottom sequence pure-exact and the diagram commutative.  The middle   vertical arrows are injections and the projective dimensions of $B, B'/B, B''/B'$ are all 1. Hence we have p.d.$B'' =1$ as well.  Furthermore, as a pure extension of $D$ by $C'$, the module $B''$ has the property that its pure extensions by $V$-modules of p.d.$\le 1$ are splitting. This means that $B'' $ is a pure-injective object in $\CC_V^*$.  Its  injective submodule $D$ has p.d.1, but, since p.d.$C' =2$, it is not tight in $B'' $, so it is not a summand in $\CC_V^*$.
\qed \end{example}

We close this section with an example and its corollary.

 \begin{example} \label{ex-pi}  An explicit example of a pure-injective object is a direct sum $C= \oplus_{i \in I} V/Vr$ for any non-unit $r \in V^\times$ and any index set $I$.  This follows from Proposition \ref{homog} (iii).
 \qed \end{example}

 Consequently, we can state the following corollary (the case for \fp\  objects has been stated above in Theorem \ref{pure}):

 \begin{corollary} \label{fppi}   Every $  V/Vr$-homogeneous $(r \in V^\times)$ torsion module  is $\Sigma$-pure-injective.
    \end{corollary}

  \begin{proof} If $D$ is any direct sum of copies of  $\partial_V$, then the submodule $D[r] $ is $ V/Vr$-homogeneous, so a direct sum of copies of $ V/Vr$.  Moreover, it is pure-injective by Example \ref{ex-pi}. Summands of $D[r]$ as well as finite direct sums of pure-injectives  are also pure-injective.
   \end{proof}

\medskip


 \section{ Cotorsion Modules}   \medskip

 We shall call an object $C$ {\it cotorsion} if  $\Ext^1_R(U, C)=0$ holds for all uniserial   (i.e. rank one) \tf\  objects $U$. Evidently, it suffices to demand this only for \cg\ uniserials.   Readers familiar with cotorsion theory immediately recognize that this cotorsion concept corresponds to  Warfield-cotorsion
 (where splitting is required for extensions by all \tf\ modules).    This claim will become even more transparent in light of the following  general statement.

  \begin{lemma} \label{cot}  An  object $C \in \CC_V^*$  is cotorsion if and only if it satisfies  the equation $\Ext^1_V(A, C)=0$ for all  \tf\ objects $A$.
  \end{lemma}

\begin{proof}  Definition settles the claim  in one direction. For the 'only if' part, assume  that $C$ is cotorsion and $A$ is   \tf\  of   p.d.$\le 1$. We know from Theorem \ref{tf1}  that  then $A$ is the union of a continuous well-ordered ascending chain of \tf\ submodules $A_\a \ (\a <\k)$ that are pure and tight in $A$ such that all the quotients $A_{\a+1}/A_\a$ are \tf\  of rank 1. We now refer to Eklof's theorem (see \cite{E}) on the extension by the union of a chain  to conclude that $A$ satisfies the quoted equation, as all the mentioned quotients in the chain satisfy it.
\end{proof}

  In order to characterize cotorsion objects in terms of solvability of systems of equations, consider a \tf\ uniserial object $U$. If it  is not cyclic, then it is generated by a countable set $\{u_n \ | \ n < \w\}$ such
that $ r_n u_{n+1} = u_n$ for some $r_n \in V^\times$ for all $n <\w$. Therefore,  an extension $B$ of module $C$ by $U$ looks like $B= \langle C, b_n\ (n < \w)\rangle $ with defining relations given by the equations  $ r_n  b_{n+1} -b_n = c_n$ for certain $c_n \in C$.  Clearly, $C$ is cotorsion if and only if $C$ is a summand of $B$ for all permissible choices of the $U$s and the $c_n$s if and only if each  consistent countable system of equations of the form
$$ r_n x_{n+1} -  x_n = c_n   \quad {\rm with}  \ c_n \in C \ (n < \w)   \leqno (6)$$
 is solvable in $C$. In the last case, a solution $x_n$ yields a complement to $C$ in $B$: the submodule generated by the elements $a_n = x_n-b_n\ (n < \w)$. This leads us to  the following theorem.

\begin{theorem}  An object $C$ is  cotorsion if and only if all consistent systems of linear equations of the form $(6)$ constructed with \tf\ uniserial objects $U$ are solvable in $C$.
\qed \end{theorem}

 More useful information about cotorsion objects is provided by the next result.

\begin{theorem} \label{cot}  {\rm (i)} All pure-injective objects in $\CC_V^*$ are also cotorsion objects.

{ \rm (ii)} Every object  is a subobject of a  cotorsion object with \tf\ cokernel.
   \end{theorem}

  \begin{proof}  (i) is an immediate consequence of the definitions,  since all extensions by \tf\ $V$-modules are pure-extensions.

  (ii) This can be verified easily, just  copy the proof of Theorem \ref{ppi}, using (6) in place of linear systems of p.d.$\le 1$, {\it mutatis mutandis}.  (By the way, the mere embedding property follows already from (i) and Theorem \ref{ppi}.)
   \end{proof}

  The next lemma is a convincing evidence  that in some respect the cotorsion objects behave like ordinary cotorsion modules, though several relevant features are missing.

 \begin{lemma} \label{cotquot}  {\rm (i)} Extension of cotorsion by cotorsion is again cotorsion.

  {\rm (ii)} Modules of  {\rm p.d.}$1$ that are epimorphic images of cotorsion objects $($modulo tight submodules$)$ are likewise cotorsion objects.
  \end{lemma}

\begin{proof} (i) is obvious.

 (ii) This is evident  considering that if $C \to C'$ is a surjective map, then for every module $A$ of p.d.$\le 1$, the induced map $\Ext^1_V(A, C) \to \Ext^1_V(A, C')$ is also surjective.
 \end{proof}

 In order to demonstrate that not all cotorsion objects are pure-injective, take e.g. a torsion object $T$
  whose first Ulm-submodule $T^1$ is  a subobject, but not divisible.  Then the embedding process mentioned in the proof  of Theorem \ref{cot}(ii) yields a cotorsion object $\overline T$ containing $T$ as a tight submodule such that $\overline  T/T$ is \tf. This $\overline T$  cannot be not pure-injective, because its Ulm-submodule contains  its torsion submodule  $T^1$ that is    a subobject, but is not injective  (cf. Theorem \ref{Ulm}).  (Another proof can be given by displaying  an extension of a pure-injective by a pure-injective (that is necessarily cotorsion) which fails to be pure-injective.)
 \medskip

 We raise the following problem on cotorsion modules in $\CC_V^*$.

 \begin{problem} Are the Ulm-submodules of cotorsion modules cotorsion and the Ulm-factors pure-injective in $\CC_V^*$ as in the case of DVD?
\qed \end{problem}

\medskip

  \nin {\bf Acknowledgment.}  We would like to thank Luigi Salce for his numerous  helpful comments.

\medskip


\nin {\footnotesize {\bf Correction.} (by L. Fuchs) Non-standard uniserial modules appear frequently in the  study of modules over \vds, we could not avoid mentioning them in our study either. I would like to correct erroneous statements on them in the literature. In her very interesting papers on non-standard uniserials (Bull. Amer. Math. Soc. {\bf 25} (1991) and Contemporary Math. {\bf 124} (1992)) B. Osofsky stated that non-standard uniserials were investigated because of their connection to Kaplansky's problem on the existence of valuation rings that are not homomorphic images of \vds. This incorrect claim (with another mistaken statement) was restated in the review of Osofsky's first article by R. G\"obel in Math. Reviews. The fact is that the problem of existence of non-standard uniserials was raised in 1980 by L. Salce during our joint investigation of modules over \vds, and it was him who named them "non-standard". S. Shelah was told  about non-standard uniserials only at the Udine Conference in April 1984 just before the night he succeeded in establishing  their existence. Then neither Shelah nor anybody else at the well-attended conference could claim that Kaplansky's problem had been solved, since  at this point nobody suspected that it was related to non-standard uniserials.  The connection became known only three months later when we solved the Kaplansky problem, and the solution relied on non-standard uniserials (see the original solution published  in  \cite{FS1}).}


\end{document}